\documentclass[11pt,a4paper,twoside]{article}
\usepackage{amsthm, amsfonts,amsmath}

\newtheorem{theorem}{Theorem}[section]
\newtheorem{corollary}{Corollary}[section]
\newtheorem{definition}{Definition}[section]
\newtheorem{example}{Example}[section]
\newtheorem{lemma}{Lemma}[section]
\newtheorem{proposition}{Proposition}[section]
\newtheorem{remark}{Remark}[section]

\newcommand\tr{\operatorname{trace}}
\def\Ric{\operatorname{Ric}}
\def\eq{\hspace*{-2.5mm}&=&\hspace*{-2.5mm}}
\def\vol{\operatorname{vol}}
\newcommand\Div{\operatorname{div}}
\def\calf{{\cal F}}

\author{Vladimir Rovenski\footnote{Department of Mathematics, University of Haifa,
 Israel
       \newline e-mail: {\tt vrovenski@univ.haifa.ac.il}
       } }

\title{On isometric immersions of sub-Riemannian manifolds}

\begin{document}

\date{}

\maketitle

\begin{abstract}
We study curvature invariants of a sub-Riemannian manifold
(i.e., a manifold with a Riemannian metric on a non-holonomic distribution) related to mutual curvature of several pairwise orthogonal subspaces of the distribution, and prove geometrical inequalities for a sub-Riemannian submanifold.
As applications, inequalities are proved for submanifolds with mutually orthogonal distributions that include scalar and mutual curvature.
For compact submanifolds, inequalities are obtained that are supported by known integral formulas for almost-product manifolds.


\end{abstract}




\section{Introduction}
\label{sec:00}

Extrinsic geometry of Riemannian submanifolds deals with properties that can be expressed
in terms of the second fundamental form and its invariants (e.g., principal curvatures).
The recent development of the
geometry of submanifolds
was inspired by the embedding theorem of J.F.\, Nash, \cite{na-1},
and theorems that surfaces with positive curvature are easily embedded in 3D space (A.D. Aleksandrov and A.V.~Pogorelov), while surfaces with negative curvature usually do not allow such an embedding (D. Gilbert and N.V.~Efimov).
This led to the following
problem (see \cite[Problem~2]{chen2}):
\textit{find a simple optimal connection between intrinsic and extrinsic invariants of a Riemannian submanifold}.
The~difficulty was to understand smooth submanifolds (the problem is different for $C^1$-immersions, \cite{na-2}) of large codimension
using only a few known relationships (fundamental Gauss-Codazzi-Ricci equations) between intrinsic and extrinsic geometry.
In 1968, S.S.\,Chern posed a question on other obstacles for a Riemannian manifold to admit an isometric minimal immersion in a Euclidean space. To study these questions, it is necessary to introduce new types of Riemannian invariants, and to find optimal relations between them and extrinsic invariants of submanifolds.

In 1990s,  B.Y.~Chen introduced the concept of $\delta$-curvature invariants for a Riemannian manifold
and proved the optimal inequality for a submanifold that involves these invariants and the square of mean curvature,
e.g.,~\cite{chen1}, the equality case led to the notion of ``ideal immersions" (isometric immersions of least possible tension).
The $\delta$-invariants are obtained from the scalar curvature
(which is the ``sum" of sectional curvatures)
by discarding some of sectional curvatures.
Similar scalar invariants are known for K\"{a}hler, contact and affine manifolds, warped products and submersions, see \cite{chen1,chen-b}.
For manifolds endowed with nonholonomic distributions or foliations, such curvature invariants have hardly been studied.

Distributions on a manifold, i.e., sub-bundles of the tangent bundle, arise in differential geometry
in terms of line fields, submersions, Lie groups actions, and almost product manifolds.
A nonholonomic manifold, i.e., a pair $(M,{\cal D})$, where ${\cal D}$ is a distribution on a smooth manifold $M$,
was introduced for the geometric interpretation of constrained systems in
classical mechanics and thermodynamics.
A~{sub-Riemannian manifold}, that is $(M,{\cal D})$ equipped with a Riemannian metric $g$ on ${\cal D}$, is
a certain type of generalization of a Riemannian manifold.
There are several lines of research in sub-Riemannian geometry based on optimal control methods, partial differential equations and constrains of other geometries, see~\cite{bf,CC-2009}.

In \cite{r-w-2022}, we introduced curvature invariants (different from $\delta$-invariants by Chen)
of a Riemannian mani\-fold equipped with complementary orthogonal distributions,
and proved the geometric inequality for submanifolds that includes our curvature invariants and the square of mean curvature.
These curvature invariants are related with the mixed scalar curvature -- a well-known curvature invariant of a
Riemannian almost $k$-product manifold, in particular, (multiply) twisted or warped products, e.g., \cite{Rov-Wa-2021}.
In \cite{r1-2022} we introduced invariants of a Riemannian manifold more general than in \cite{r-w-2022}, related to the mutual curvature of noncomplementary pairwise orthogonal subspaces of the tangent bundle.
In the case of one-dimensional subspaces, the mutual curvature is equal to half the scalar curvature of the subspace spanned by them,
and in the case of complementary subspaces, this is the mixed scalar curvature.
Using these invariants, we proved
inequalities for Riemannian submanifolds
and gave applications for sub-Riemannian submanifolds.

In this article, we study curvature invariants (defined in \cite{r-w-2022,r1-2022})
and also introduce Chen-type invariants for a sub-Riemannian manifold.
We~prove geometrical inequalities for submanifolds with mutually orthogonal distributions
that include scalar and mutual curvature.
In the case of
compact submanifolds, we obtain the inequalities supported by known integral formulas for almost-product manifolds.


The article is organized as follows.
In Section~\ref{sec:01a} (following the introductory Section~\ref{sec:00}),
we recall some integral formulas containing scalar and mutual curvature for a sub-Riemannian manifold.
In Section~\ref{sec:01b}, we introduce and study scalar invariants based on this kind of curvature.
In Section~\ref{sec:02}, we prove geometric inequalities for a sub-Riemannian submanifold equipped with distributions.

\section{The mutual curvature of distributions}
\label{sec:01a}

Here, we recall definitions of mutual curvature and mixed scalar curvature of distributions on a sub-Riemannian manifold
and briefly discuss equalities with them and divergence of some vector fields, which lead to integral formulas on a compact manifold.

Let an $\,n$-dimensional Riemannian manifold $(M,g)$ with the Levi-Civita connection $\nabla$
be endowed with a  $d$-dimensional distribution ${\cal D}$ (subbundle of the tangent bundle $TM$ of rank~$d$).
The Riemannian curvature tensor is given by $R_{X,Y} = [\nabla_{X}, \nabla_{Y}] -\nabla_{[X,\, Y]}$,
its contraction is the Ricci tensor $\Ric_{X,Y}=\tr( Z\mapsto R_{Z,X}\,Y )$,
and the trace of Ricci tensor is the~scalar curvature $\tau=\tr_g\Ric$, e.g., \cite{pet}.

Let ${\cal D}^\bot$ be the orthogonal complement to ${\cal D}$ in $TM$, its rank is $d^\bot=n-d$.
We call $(M,g,{\cal D},{\cal D}^\bot)$ a Riemannian almost product manifold, see \cite{g1967}.
The second fundamental form $h$ and integrability tensor $T$
of ${\cal D}$ (and, similarly, tensors $h^\bot$ and $T^\bot$ of ${\cal D}^\bot$) are defined as follows:
\begin{eqnarray*}
 h(X,Y)=\frac12\,(\nabla_X Y+\nabla_Y X)^\perp,\quad
 T(X,Y)=\frac12\,(\nabla_X Y-\nabla_Y X)^\perp.
\end{eqnarray*}
If ${\cal D}$ is integrable (i.e., $T=0$), then it is tangent to a foliation.
Denote by $H=\tr_g h$ and $H^\perp=\tr h^\bot$ the mean curvature vectors of ${\cal D}$ and ${\cal D}^\bot$, respectively.
We call ${\cal D}$ \textit{totally geodesic} if $h=0$,
\textit{harmonic} if $H=0$ and \textit{totally umbilical} if $h = (H/d)\,g$ (and similarly, for ${\cal D}^\bot$).

Let $\{e_i\}$ be an adapted local orthonormal frame, i.e.,
 $\{e_1,\ldots, e_{d}\}\subset{{\cal D}}$ and $\{e_{d+1},\ldots, e_{n}\}\subset{\cal D}^\bot$.
The~mixed scalar curvature ${\rm S}_{\,\rm mix}({\cal D},{\cal D}^\bot)$ is a function on $M$ defined by
\[
 {\rm S}_{\,\rm mix}({\cal D},{\cal D}^\bot) = \sum\nolimits_{\,1\le a\le d,\ d<b\le n} K(e_a\wedge {e}_b),
\]
where $K(e_a\wedge\,{e}_b)=g(R_{e_a,{e}_b}\,e_b, {e}_a)$ is the sectional curvature of the plane $e_a\wedge\,{e}_b$,
and it does not depend on the choice of~frames.
For example, if ${\cal D}$ (or ${\cal D}^\bot$) is one-dimensional and locally spanned by a unit vector field $N$, then
$S_{\rm mix}({\cal D},{\cal D}^\bot)=\Ric_{N,N}$.
The following formula for complementary orthogonal distributions ${\cal D}$ and ${\cal D}^\bot$
on a Riemannian manifold $(M,g)$ was proved in~\cite{wa1}:
\begin{equation}\label{eq-wal2}
 \Div(H+H^\bot) =
 S_{\rm mix}({\cal D},{\cal D}^\bot)+\|\,h\,\|^2+\|\,h^\perp\,\|^2-\|\,H\,\|^2-\|\,H^\perp\,\|^2-\|\,T\,\|^2-\|\,T^\perp\,\|^2 .
\end{equation}

\begin{example}\rm
Let ${\cal D}$ be tangent to a codimension one foliation $\calf$, and $N$ be a unit normal to the leaves of $\calf$.
The~shape
operator $A_N:T\calf\to T\calf$
is given~by $A_N(X)=-\nabla_X\,N$, where $\nabla$ is the Levi--Civita connection. The~{generalized mean curvatures} 
 $\sigma_r=\sigma_r(A_N)$ are functions on $M$ defined as coefficients of the $n$-th degree polynomial $\det({\rm id}_{\cal D} + \,t A_N)$ in~$t$.
Thus, $\sigma_0=1,\ \sigma_1=\tr A_N, \ \ldots, \ \sigma_n=\det A_N$.
In this case, \eqref{eq-wal2} reduces to
\begin{equation}\label{eq-wal2b}
 \Div(\nabla_N\,N+\sigma_1 N) =
 \Ric_{N,N} - \,2\,\sigma_2.
\end{equation}
\end{example}

Next, let a Riemannian manifold $(M,g)$ be endowed with three pairwise orthogonal $n_i$-dimensional distributions ${\cal D}_i\ (i=1,2,3)$
such that $TM={\cal D}_1\oplus{\cal D}_2\oplus {\cal D}_3$.
We call $(M,g,{\cal D}_1,{\cal D}_2,{\cal D}_3)$ a \textit{Riemannian almost 3-product manifold}.
Denote by ${\cal D}^\bot_i$ the orthogonal complement to ${\cal D}_i$ in $TM$, its rank is $n^\bot_i=n-n_i$.

\begin{remark}\rm
A {Riemannian almost multi-product manifold} is a Riemannian manifold equipped with $k\ge2$ pairwise orthogonal complementary distributions
${\cal D}_1,\ldots,{\cal D}_k$.
We meet this structure in such topics of differential geometry as multiply-warped (or twisted) products and the webs of foliations;
see e.g.,~\cite{rov-IF-k}.
In~particular, almost 3-product manifolds appear naturally among
almost para-$f$-manifolds, lightlike manifolds, orientable 3-manifolds (since they admit 3 linearly independent vector~fields),
webs composed of 3 generic foliations,
minimal hypersurfaces in space forms with 3 distinct principal curvatures,
tubes over standard embeddings of a projective plane in a sphere, etc.
\end{remark}

The {second fundamental forms} ${h}_i:{\cal D}_i\times {\cal D}_i\to{\cal D}^\bot_i$ and
the {integrability tensors} ${T}_i:{\cal D}_i\times {\cal D}_i\to{\cal D}^\bot_i$ of ${\cal D}_i$ (and similarly, $h^\bot_i$ and $T^\bot_i$ of orthogonal distributions ${\cal D}^\bot_i$)
are defined by
\begin{equation*}
 2\,h_i(X,Y)= (\nabla_X Y+\nabla_Y X)^\bot,\quad 2\,T_i(X,Y) = (\nabla_X Y-\nabla_Y X)^\bot = [X, Y]^\bot .
\end{equation*}
Then $H_i=\tr_g h_i$ is called the mean curvature vector field of the distribution ${\cal D}_i$.
 A distribution ${\cal D}_i$ is integrable (or involutive) if $T_i=0$,
and ${\cal D}_i$ is \textit{totally umbilical}, \textit{minimal}, or \textit{totally geodesic},
if  ${h}_i=({H}_i/n_i)\,g,\ {H}_i=0$, or ${h}_i=0$, respectively.
Let $x\in M$ and $\{e_i\}$ be an adapted orthonormal frame on the subspace ${\cal D}_1(x)\oplus{\cal D}_2(x)$, i.e.,
 $\{e_1,\ldots, e_{n_1}\}\subset{{\cal D}_1(x)},\ \{e_{n_{1}+1},\ldots, e_{d}\}\subset{{\cal D}_2(x)}$.
The~\textit{mutual curvature} of a pair $({\cal D}_1,{\cal D}_2)$ is a function on~$M$ defined by
\[
 {\rm S}_{\,\rm m}({\cal D}_1(x),{\cal D}_2(x)) = \sum\nolimits_{\,a\le n_1,\ n_1<b\le d} K(e_a,{e}_b),
\]
and it does not depend on the choice of~frames, e.g.,~\cite{rst-40}.
The~mixed scalar curvature of the triple $({\cal D}_1,{\cal D}_2,{\cal D}_3)$ is defined similarly as
${\rm S}_{\,\rm mix}({\cal D},{\cal D}^\bot)$ for a pair $({\cal D},{\cal D}^\bot)$,
and it can be presented as follows, e.g., \cite{rov-IF-k}:
\begin{equation}\label{E-Smix-3a}
 {\rm S}_{\,\rm mix}({\cal D}_1,{\cal D}_2,{\cal D}_3)={\rm S}_{\,\rm m}({\cal D}_1,{\cal D}_2)+{\rm S}_{\,\rm m}({\cal D}_1,{\cal D}_3)+{\rm S}_{\,\rm m}({\cal D}_2,{\cal D}_3).
\end{equation}

\begin{lemma}\label{L-Qi}
The following formulas are true:
\begin{eqnarray}\label{eq-IF-3}
 && 2\,{\rm S}_{\,\rm mix}({\cal D}_1,{\cal D}_2,{\cal D}_3) = \Div(H_1+H_1^\bot+H_2+H_2^\bot+H_3+H_3^\bot) -Q_{1}-Q_{2}-Q_{3} , \\
\label{eq-IF-l2}
 && S_{\rm m}({\cal D}_1,{\cal D}_2) = \Div(H_1+H_1^\bot+H_2+H_2^\bot-H_3-H_3^\bot) -Q_{1}-Q_{2}+Q_{3} ,
\end{eqnarray}
where \,$Q_{i}=\|\,h_i\,\|^2+\|\,h_i^\perp\,\|^2-\|\,H_i\,\|^2-\|\,H_i^\perp\,\|^2-\|\,T_i\,\|^2-\|\,T_i^\perp\,\|^2\,$ for $\,i=1,2,3$.
\end{lemma}

\begin{proof}
We can write \eqref{E-Smix-3a} in the form
\begin{equation*}
 2\,{\rm S}_{\,\rm mix}({\cal D}_1,{\cal D}_2,{\cal D}_3) = {\rm S}_{\,\rm mix}({\cal D}_1,{\cal D}_1^\bot)
 +{\rm S}_{\,\rm mix}({\cal D}_2,{\cal D}_2^\bot) +{\rm S}_{\,\rm mix}({\cal D}_3,{\cal D}_3^\bot),
\end{equation*}
or in the form (expressing mutual curvature in terms of mixed scalar curvature)
\begin{eqnarray*}
 {\rm S}_{\,\rm m}({\cal D}_1,{\cal D}_2) \eq {\rm S}_{\,\rm mix}({\cal D}_1,{\cal D}_2,{\cal D}_3)
 - 2\,{\rm S}_{\,\rm mix}({\cal D}_1\oplus{\cal D}_2, {\cal D}_3) \\
 \eq {\rm S}_{\,\rm mix}({\cal D}_1,{\cal D}_1^\bot)
 +{\rm S}_{\,\rm mix}({\cal D}_2,{\cal D}_2^\bot) -{\rm S}_{\,\rm mix}({\cal D}_3,{\cal D}_3^\bot).
\end{eqnarray*}
Thus, using \eqref{eq-wal2}, we get \eqref{eq-IF-3} and \eqref{eq-IF-l2}.
\end{proof}


\begin{example}\label{Ex-2.4}\rm
Let $(M^3,g)$ admit three pairwise orthogonal codimension-one foliations ${\mathcal F}_i$, and let $N_i$ be unit vector fields orthogonal to ${\mathcal F}_i$.
Writing down \eqref{eq-wal2b} for each $N_i$, summing for $i=1,2,3$, and using the equality
 $\tau=\sum\nolimits_{\,i=1}^3{\rm Ric}_{\,N_i,N_i}$, where $\tau$ is the scalar curvature of~$(M,g)$,
yields the {formula}
\[
 \Div\sum\nolimits_{\,i=1}^3(\nabla_{N_i}\,N_i+\sigma_1(A_{N_i}) N_i) =
 2\sum\nolimits_{\,i=1}^3\sigma_2({\mathcal F}_i)-\tau\,.
\]
 Two
consequences, using $2\,\sigma_2(A_{N_i})=\tr(A_{N_i})^2-(\tr A_{N_i})^2$\,:

\quad
$\bullet$~if $\tau<0$ then each foliation ${\mathcal F}_i$ cannot be totally umbilical;

\quad
$\bullet$~if $\tau>0$ then each foliation ${\mathcal F}_i$ cannot be harmonic.
\end{example}

\begin{remark}\label{R-3.5}\rm
Applying the Divergence Theorem to \eqref{eq-wal2}, \eqref{eq-wal2b}, \eqref{eq-IF-3} and \eqref{eq-IF-l2} on a compact Riemannian manifold gives well-known integral formulas.
These formulas
can be extended for distributions
defined on the complement $M\setminus\Sigma$ of a union $\Sigma$ of finitely many closed codimension
$k\ge 2$
submanifolds of a manifold~$M$.
Namely,
\textit{if $(M,g)$ is a closed oriented Riemannian manifold, $X$ is a vector field on an open set $M\setminus\Sigma$, $(k-1)(p-1)\ge 1$
and $\|\,X\|\,\in L^p(M,g)$, then $\int_M\,(\Div X)\,d\vol_g = 0$}, see \cite{wa2} and \cite[p.~75]{Rov-Wa-2021}.
\end{remark}



\section{Invariants based on scalar and mutual curvature}
\label{sec:01b}

Here, we introduce and study scalar invariants based on scalar and mutual curvature.

Given integer $k\ge2$, let $V_1,\ldots,V_k$ be mutually orthogonal subspaces of ${\cal D}_x$ at a point $x\in M$
with $\dim V_i=n_i\ge1$.
Let $\{e_i\}$ be an adapted orthonormal basis of the subspace $V=\bigoplus_{\,i=1}^{\,k} V_i$, i.e.,
 $\{e_1,\ldots, e_{n_1}\}\subset V_1,
 \ldots,
 \{e_{n_{k-1}+1},\ldots, e_{n_k}\}\subset V_k$.
Define the \textit{mutual curvature} of the set $\{V_1,\ldots,V_k\}$~by
\begin{equation*}
 {\rm S}_{\,\rm m}(V_1,\ldots,V_k)
 = \sum\nolimits_{\,i<j}\sum\nolimits_{\,n_{i-1}<a\,\le n_i,\ n_{j-1}<b\le n_j} K(e_a\wedge\,{e}_b).
\end{equation*}
Note that ${\rm S}_{\,\rm m}(V_1,\ldots,V_k)$ does not depend on the choice of frames.
We immediately have
\[
 {\rm S}_{\,\rm m}(V_1,\ldots,V_k) = \sum\nolimits_{\,i<j} {\rm S}_{\,\rm m}(V_i, V_j),
\]
where
 ${\rm S}_{\,\rm m}(V_i, V_j) = \sum\nolimits_{\,n_{i-1}<a\,\le n_i,\ n_{j-1}<b\le n_j} K(e_a\wedge\,{e}_b)$.

For the scalar curvature $\tau(V)=\tr_g\Ric|_{\,V}$ (the trace of the Ricci tensor on a subspace $V=\bigoplus_{\,i=1}^{\,k} V_i$)
and the scalar curvatures $\tau(V_i)=\tr_g\Ric|_{\,V_i}$ of subspaces $V_i$ we get
\begin{equation}\label{E-Smix-3}
 \tau(V) = 2\,{\rm S}_{\,\rm m}(V_1,\ldots, V_k) +\sum\nolimits_{\,i=1}^k \tau(V_i)\,.
\end{equation}
For example, if all subspaces $V_i$ are one-dimensional, then $2\,{\rm S}_{\,\rm m}(V_1,\ldots, V_k)=\tau(V)$.

For an integer $k\ge2$, denote by $S(d,k)$ the set of unordered $k$-tuples $(n_1,\ldots,n_k)$ of natural numbers satisfying
$n_1+\ldots+n_k\le d$.
Denote by $S(d)$ the set of all unordered $k$-tuples with~$k\ge2$ and $n_1+\ldots+n_k\le d$.

\begin{definition}[\rm\cite{r1-2022}]\rm
For a $k$-tuple $(n_1,\ldots,n_k)\in S(d,k)$ the scalar invariants $\delta^\pm_{{\rm m},{\cal D}}(n_1,\ldots, n_k)$
are defined~by
\begin{equation*}
 \delta^+_{{\rm m},{\cal D}}(n_1,\ldots, n_k)(x) = \max {\rm S}_{\,\rm m}(V_1,\ldots,V_k),\quad
 \delta^-_{{\rm m},{\cal D}}(n_1,\ldots, n_k)(x) = \min {\rm S}_{\,\rm m}(V_1,\ldots,V_k),
\end{equation*}
where $V_1,\ldots,V_k$ run over all $k$ mutually orthogonal subspaces of ${\cal D}_x$ with $\dim V_i=n_i\ (i=1,\ldots,k)$.
For~${\cal D}=TM$ we get invariants $\delta^\pm_{\rm m}(n_1,\ldots, n_k)=\delta^\pm_{{\rm m},TM}(n_1,\ldots, n_k)$, see also \cite{r-w-2022}.
\end{definition}

If the sectional curvature of $(M, g)$ along ${\cal D}$ satisfies $c\le K_{\,|\,{\cal D}}\le C$ and $\sum\nolimits_{\,i=1}^k n_i = s\le d$, then
\begin{eqnarray*}
  \frac c2\,(s^2-\sum\nolimits_{\,i} n_i^2)=
  c\sum\nolimits_{\,i<j} n_i\,n_j \le
  \delta^-_{{\rm m},{\cal D}}(n_1,\ldots n_k)\le \\
  \le \delta^+_{{\rm m},{\cal D}}(n_1,\ldots n_k) \le C\sum\nolimits_{\,i<j} n_i\,n_j
 =\frac C2\,(s^2-\sum\nolimits_{\,i} n_i^2)\,.
\end{eqnarray*}

\begin{example}\rm
Recall that for a subspace $V$ spanned by $q+1$ orthonormal vectors $\{e_0,e_1,\ldots,e_{q}\}$ of $(M, g)$, the \textit{$q$-th Ricci curvature} is
 $\Ric_{\,q}(V)=\sum\nolimits_{\,i=1}^{q}\,K(E_0,E_i)$,
e.g.,~\cite{r-98}.
For $k=2$ and $n_1=1$, using the intermediate Ricci curvature, we get
$\delta^+_{\rm m}(1, n_2)(x)=\max \Ric_{n_2}(V)$
 and 
 $\delta^-_{\rm m}(1, n_2)(x)=\min \Ric_{n_2}(V)$,
where
$V=span(V_1,V_2)$ and
$V_1,V_2$ run over all mutually orthogonal subspaces of ${\cal D}_x$ such that $\dim V_1=1$ and $\dim V_2=n_2$.
\end{example}

For a $k$-tuple $(k\ge0)$ and $x\in M$, B.-Y~Chen \cite[Sect.~13.2]{chen1} defined the following curvature invariants:
\begin{eqnarray}\label{E-Chen-delta}
\nonumber
 2\,\delta(n_1,\ldots, n_k)(x)=\tau(x)-\min\,\{\tau(V_1)+\ldots +\tau(V_k)\} ,\\
 2\,\hat\delta(n_1,\ldots, n_k)(x)=\tau(x)-\max\,\{\tau(V_1)+\ldots +\tau(V_k)\},
\end{eqnarray}
where $V_1,\ldots,V_k$ run over all $k$ mutually orthogonal subspaces of $T_x M$ with $\dim V_i=n_i\ (i=1,\ldots,k)$.
The~coefficient 2 in \eqref{E-Chen-delta} is due to the definition of the scalar curvature in \cite{chen1} as half of the ``trace~Ricci".


\begin{definition}[\rm\cite{r1-2022}]\rm
For each $k$-tuple $(k\ge0)$ and $x\in M$, we define Chen-type $\delta_{\cal D}$-invariants of $(M,g;{\cal D})$
by
\begin{eqnarray}\label{E-ineq1-D}
\nonumber
 2\,\delta_{\cal D}(n_1,\ldots, n_k)(x)=\tau({\cal D}_x)-\min\,\{\tau(V_1)+\ldots +\tau(V_k)\} ,\\
 2\,\hat\delta_{\cal D}(n_1,\ldots, n_k)(x)=\tau({\cal D}_x)-\max\,\{\tau(V_1)+\ldots +\tau(V_k)\}\,,
\end{eqnarray}
where $V_1,\ldots,V_k$ run over all $k$ mutually orthogonal subspaces of ${\cal D}_x$ with $\dim V_i=n_i\ (i=1,\ldots,k)$.
\end{definition}

The theory of $\delta_{\cal D}$-invariants \eqref{E-ineq1-D} of a sub-Riemannian manifold can be developed similarly to the theory of Chen's $\delta$-invariants of a Riemannian manifold.

The~$\delta^\pm_{{\rm m},{\cal D}}$-invariants are related with the curvature invariants in \eqref{E-ineq1-D} by the following inequalities.

\begin{proposition}
Let $k\ge2$. If $n_1+\ldots+n_k<d$, then the following inequalities are valid:
\begin{eqnarray}\label{E-ineq1}
\nonumber
  \delta^+_{{\rm m},{\cal D}}(n_1,\ldots, n_k) \ge \delta_{\cal D}(n_1,\ldots, n_k) -\delta_{\cal D}(n_1+\ldots+n_k)\,,\\
  \delta^-_{{\rm m},{\cal D}}(n_1,\ldots, n_k) \le \hat\delta_{\cal D}(n_1,\ldots, n_k) -\hat\delta_{\cal D}(n_1+\ldots+n_k)\,,
\end{eqnarray}
and
if $n_1+\ldots+n_k{=}d$, then
$\hat\delta_{\cal D}(n_1,\ldots, n_k){=}\delta^-_{{\rm m},{\cal D}}(n_1,\ldots, n_k)\le
\delta^+_{{\rm m},{\cal D}}(n_1,\ldots, n_k){=}\delta_{\cal D}(n_1,\ldots, n_k)$.
In particular,
if $n_1+\ldots+n_k=d-1$, then
 $\hat\delta_{\cal D}(n_1,\ldots, n_k)-\min\Ric_{d-1}({\cal D})\ge \delta^-_{{\rm m},{\cal D}}(n_1,\ldots, n_k)\ge
 \delta^+_{{\rm m},{\cal D}}(n_1,\ldots, n_k)\ge\delta_{\cal D}(n_1,\ldots, n_k)-\max\Ric_{d-1}({\cal D})$.
\end{proposition}

\begin{proof}
Using \eqref{E-Smix-3} and the equality $-\min a = \max(-a)$, we get
\begin{eqnarray*}
 && 2\,\delta_{\cal D}(n_1,\ldots, n_k)(x) =\tau(V)-\min\,\{\tau(V_1)+\ldots +\tau(V_k)\} \\
 &&= \tau({\cal D}_x) +\max(\tau_k(x)-(\tau(V_1)+\ldots +\tau(V_k)) -\tau_k(V)  ) \\
 && \le \tau({\cal D}_x) -\min \tau_k(x) + 2\max {\rm S}_{\,\rm m}(V_1,\ldots, V_k) \\
 && = 2\,\delta_{\cal D}(n_1+\ldots+n_k)(x) + 2\,\delta^+_{{\rm m},{\cal D}}(n_1,\ldots, n_k)(x),
\end{eqnarray*}
hence, \eqref{E-ineq1}$_1$ is valid. The proof of \eqref{E-ineq1}$_2$ is similar.
The case of $n_1+\ldots+n_k=d$ follows from \eqref{E-ineq1}.
The case of $n_1+\ldots+n_k=d-1$ follows from $\delta_{\cal D}(d-1)(x)=\max\Ric_{d-1}({\cal D}_x)$
and $\hat\delta_{\cal D}(d-1)(x)=\min\Ric_{d-1}({\cal D}_x)$.
\end{proof}

\begin{corollary}
If $(M,g;{\cal D})$ has nonnegative sectional curvature along ${\cal D}$ and $k\ge2$, then
\[
 \hat\delta(n_1,\ldots, n_k) \le
 \delta^-_{{\rm m},{\cal D}}(n_1,\ldots, n_k)\le
 \delta^+_{{\rm m},{\cal D}}(n_1,\ldots, n_k)\le\delta(n_1,\ldots, n_k),\
\]
and if $(M,g;{\cal D})$ has nonpositive sectional curvature along ${\cal D}$, then the inequalities are opposite.
\end{corollary}

\section{Geometric inequalities for a submanifold 
with distributions}
\label{sec:02}

First, we consider \textit{adapted} isometric immersions $f: (M,g; {\cal D}) \to (\bar M,\bar g; \bar{\cal D})$
of sub-Riemannian manifolds, i.e., $f_*({\cal D})\subset \bar{\cal D}_{\,|\,f(M)}$.
If
${\cal D}$ and $\bar{\cal D}$ are the sums of $s\ge2$ mutually orthogonal distributions,
i.e., ${\cal D}=\bigoplus_{i=1}^s{\cal D}_i$ and $\bar{\cal D}=\bigoplus_{i=1}^s\bar{\cal D}_i$,
then we also require the following: $f_*({\cal D}_i)\subset \bar{\cal D}_{i\,| f(M)}$ for all $i$.
Below we assume $s=2$.

\begin{remark}\rm
A sub-Riemannian structure on a smooth manifold $M$ can be obtained from a special immersion of $M$ in $(\bar M,\bar g; \bar{\cal D})$.
Namely, let $f_*(TM)$ intersects transversally with the distribution $\bar{\cal D}$ restricted to $f(M)$, then
$f: M \to \bar M$ induces a required distribution ${\cal D}=f_*^{-1}(\bar{\cal D}\cap f(TM))$ on $M$ with induced metric $g$.
\end{remark}

We will identify $M$ with its image $f(M)$ (since the induced metric on $f(M)$ is equal to $g$) and put a top ``bar" for objects related to $\bar M$.
Let $TM^\bot$ be the normal bundle of the submanifold $M\subset\bar M$
and $\bar{h}:TM\times TM\to TM^\bot$ be the second fundamental form of $M$.
Recall the Gauss equation for an isometric immersion $f$,
e.g., \cite{chen1}:
\begin{equation}\label{E-Gauss-class}
  \bar g(\bar R_{Y,Z}\,U,X) = g(R_{Y,Z}\,U,X) + \bar g(\bar{h}(Y,U), \bar{h}(Z,X)) -\bar g(\bar{h}(Z,U), \bar{h}(Y,X)) ,
  \quad U,X,Y,Z\in TM,
\end{equation}
where $\bar R$ and $R$ are the curvature tensors of $(\bar M,\bar g)$ and $(M,g)$, respectively.
The mean curvature vector of a subspace $V\subset{\cal D}_x$ is given by
$\bar{H}_V = \tr \bar{h}_{\,|\,V}= \sum_i \bar{h}(e_i,e_i)$, where $e_i$ is an orthonormal basis of~$V$.
Thus, $\bar{H}_{\cal D} = \tr_g \bar{h}_{\,|\,{\cal D}}$ is the mean curvature vector of~${\cal D}$,
and $\bar{H} = \tr_g \bar{h}$ is the mean curvature vector of $M$.
An~isometric immersion $f$ with the property $\bar{H}_{\cal D}=0$ is called ${\cal D}$-minimal
(minimal if $\bar{H}=0$).
Set
\begin{equation*}
 {\cal H}_x(s)
 = \max \{\,\|\,\bar{H}_V\,\|\,:
 \ V\subset {\cal D}_x,
 \ \dim V=s>0\}.
\end{equation*}
If $s=d$,
then $\bar{H}_V=\bar{H}_{{\cal D}_x}$.
For $s < d$ the condition
${\cal H}(s)=0$ implies that $\bar{h}_{\,|\,{\cal D}}=0$.

An~isometric immersion $f: (M,g; {\cal D}) \to (\bar M,\bar g)$
is called \textit{mixed totally geodesic} on $V=\bigoplus_{\,i=1}^{\,k} V_i\subset{\cal D}$ if
 $\bar{h}(X,Y)=0$
 for all
 $X\in{V}_i,\,Y\in{V}_j$ and $i\ne j$\,.

\begin{theorem}\label{T-D}
Let $f: (M,g; {\cal D})\to(\bar M,\bar g,\bar{\cal D})$ be an adapted isometric immersion, and $\sum_{\,i}n_i=s\le d$.
Then
\begin{eqnarray}\label{E-ineq-D}
 \delta^+_{{\rm m},{\cal D}}(n_1,\ldots,n_k) \le
 \bar\delta^+_{{\rm m},\bar{\cal D}}(n_1,\ldots,n_k) +
 \frac{k-1}{2\,k}\left\{\begin{array}{cc}
  {\cal H}_{\cal D}(s)^2,  & {\rm if}\ s<d, \\
  \|\,\bar{H}_{\cal D}\,\|^2, & {\rm if}\ s=d.
\end{array}\right.
\end{eqnarray}
The equality in \eqref{E-ineq-D} holds at a point $x\in M$ if and only if
there exist mutually orthogonal subspaces ${V}_1, \ldots, {V}_k$ of ${\cal D}_x$ with $\sum_{\,i}n_i=s$ such that
$f$ is mixed totally geodesic on $V=\bigoplus_{\,i=1}^{\,k} V_i$, $\bar{H}_1=\ldots=\bar{H}_k$, \ $\|\,\bar{H}_{V}\|\,={\cal H}_{{\cal D}_x}(s)$
and $\bar{\rm S}_{\,\rm m}({V}_1, \ldots, {V}_k)=\bar\delta^+_{{\rm m},\bar{\cal D}}(n_1,\ldots,n_k)(x)$.
\end{theorem}

\begin{proof}
Taking trace of the Gauss equation \eqref{E-Gauss-class} for the immersion $f$ along $V$ and $V_i$ yields the equalities
\begin{equation}\label{E-Si}
  \bar\tau(V) - \tau(V) = \|\,\bar{h}_{V}\,\|^2 - \|\,\bar{H}_{V}\,\|^2, \qquad
  \bar\tau({V}_i) - \tau({V}_i) = \|\,\bar{h}_i\,\|^2 - \|\,\bar{H}_i\,\|^2,
\end{equation}
where $\bar\tau(V)$, $\bar\tau({V}_i)$ and $\tau(V)$, $\tau({V}_i)$ are the scalar curvatures of subspaces $V=\bigoplus_{\,i=1}^{\,k} V_i$ and ${V}_i$ for the curvature tensors $\bar R$ and $R$, respectively,
$\bar{h}_i$ and $\bar{H}_i$ are the second fundamental form and mean curvature vector of~$V_i$.

 Assume that $\bar{H}_V\ne0$ is satisfied on an open set $U\subset M$ and
complement over $U$ an adapted local orthonormal frame $\{e_1,\ldots,e_n\}$ of $(M,g)$ with vector $e_{n+1}$ parallel to $\bar{H}_V$.
Using $\bar{H}_V=\sum_{\,i=1}^{\,k} \bar{H}_i$ and the algebraic inequality $a_1^2+\ldots+a_k^2\ge \frac1k\,(a_1+\ldots+a_k)^2$
for real $a_i=\bar g(\bar{H}_i, e_{n+1})$, we find
\begin{equation}\label{E-Si3}
  \sum\nolimits_{\,i}\|\,\bar{H}_i\,\|^2 \ge \sum\nolimits_{\,i} \bar g(\bar{H}_i, e_{n+1})^2
  \ge \frac1k\,\|\,\bar{H}_V\,\|^2,
\end{equation}
and the equality holds if and only if $\bar{H}_1=\ldots=\bar{H}_k$.
The above inequality is trivially satisfied for $\bar{H}_V=0$, hence it is valid on $M$.
Set
$\|\,\bar{h}^{\rm mix}_{ij}\,\|^2=\sum_{\,e_a\in{V}_i,\, e_b\in{V}_j} \|\,\bar{h}(e_a,e_b)\,\|^2$ for $i\ne j$
and note that
\begin{equation}\label{E-Si4}
 \|\,\bar{h}_V\,\|^2 = \sum\nolimits_{\,i}\|\,\bar{h}_i\,\|^2 +\sum\nolimits_{\,i<j}\|\,\bar{h}^{\rm mix}_{ij}\,\|^2
 \ge \sum\nolimits_{\,i}\|\,\bar{h}_i\,\|^2 ,
\end{equation}
and the equality holds if and only if $\|\,\bar{h}^{\rm mix}_{ij}\,\|^2=0\ (\forall\,i<j)$,
i.e., $f$ is mixed totally geodesic along $V$.

By \eqref{E-Si}, \eqref{E-Si3}, \eqref{E-Si4} and the equalities
\begin{equation*}
 \bar\tau(V) = 2\,\bar{\rm S}_{\,\rm m}({V}_1,\ldots,{V}_k) +\sum\nolimits_{\,i}\bar\tau({V}_i),\quad
 \tau(V) = 2\,{\rm S}_{\,\rm m}({V}_1,\ldots,{V}_k) +\sum\nolimits_{\,i}\tau({V}_i),
\end{equation*}
see \eqref{E-Smix-3}, we obtain
\begin{eqnarray*}
\nonumber
 && 2\,{\rm S}_{\,\rm m}({V}_1,\ldots,{V}_k) =  2\,\bar{\rm S}_{\,\rm m}({V}_1,\ldots,{V}_k)
 + \sum\nolimits_{\,i} (\bar\tau({V}_i) - \tau({V}_i)) + \|\,\bar{H}_V\,\|^2 - \|\,\bar{h}_V\,\|^2 \\
\nonumber
 && \le 2\,\bar\delta^+_{{\rm m},\bar{\cal D}}(n_1,\ldots,n_k) - ( \|\,\bar{h}_V\,\|^2 - \sum\nolimits_{\,i}\|\,\bar{h}_i\,\|^2 )
  + (\|\,\bar{H}_V\,\|^2 - \sum\nolimits_{\,i}\|\,\bar{H}_i\,\|^2 ) \\
 && \le 2\,\bar\delta^+_{{\rm m},\bar{\cal D}}(n_1,\ldots,n_k) + \frac{k-1}k\,{\cal H}(s)^2 ,
\end{eqnarray*}
(and the equality holds in the second line if and only if
$\bar{\rm S}_{\,\rm m}({V}_1,\ldots,{V}_k)=\bar\delta^+_{{\rm m},\bar{\cal D}}(n_1,\ldots,n_k)$
and $\|\,\bar{H}_V\|\,={\cal H}_x(s)$ at each point $x\in M$)
that proves \eqref{E-ineq-D} for $s<d$.
The case $\sum_{\,i}n_i=d$ of \eqref{E-ineq-D} was proved in \cite{r-w-2022}.
\end{proof}

\begin{corollary}
For an adapted isometric immersion $f:(M,g; {\cal D})\mapsto(\bar M,\bar g; \bar{\cal D})$ with sectional curvature along $\bar{\cal D}$
bounded above by~$c$
and $\sum_{\,i}n_i=s\le d$, from \eqref{E-ineq-D} we
get the following inequality:
\[
 \delta^+_{{\rm m},{\cal D}}(n_1,\ldots,n_k) \le \frac c2\,(d^2-\sum\nolimits_{\,i} n_i^2)
 +\frac{k-1}{2\,k}\left\{\begin{array}{cc}
  {\cal H}_{\cal D}(s)^2,  & {\rm if}\ s<d, \\
  \|\,\bar{H}_{\cal D}\,\|^2, & {\rm if}\ s=d.
\end{array}\right.
\]
\end{corollary}

\begin{corollary}[\rm see {\rm \cite[Corollary 4]{r1-2022}}]
A sub-Riemannnian manifold $(M,g; {\cal D})$ with the condition $\delta^+_{{\rm m},{\cal D}}(n_1,\ldots,n_k)>0$
for some $(n_1,\ldots,n_k)\in S(d,k)$ such that $\sum_{\,i}n_i=d$, does not admit ${\cal D}$-minimal isometric immersions in a Euclidean space.
\end{corollary}

\begin{proof}
This follows directly from \eqref{E-ineq-D}.
\end{proof}

\begin{corollary}\label{C-new2}
Let $f: (M,g; {\cal D})\to(\bar M,\bar g)$ be an isometric immersion.
If $M$ is compact with ${\cal D}$ defined on an open set $M\setminus\Sigma$, $(k-1)(p-1)\ge 1$ and $\|\,H+H^\bot\|\,\in L^p(M,g)$
(see Remark~\ref{R-3.5}),~then
\begin{eqnarray*}
 \int_M (\|\,H\,\|^2+\|\,H^\perp\,\|^2+\|\,T\,\|^2+\|\,T^\perp\,\|^2-\|\,h\,\|^2-\|\,h^\perp\,\|^2)\,d\vol_g \le \\
 \le \frac{1}{4}\int_M \|\,\bar{H}\,\|^2\,d\vol_g +\,\bar\delta^+_{{\rm m}}(d,d^\bot)\,{\rm Vol}(M,g).
\end{eqnarray*}
\end{corollary}

\begin{proof}
Applying \eqref{eq-wal2} to \eqref{E-ineq-D} and the Divergence-type theorem in Remark~\ref{R-3.5}, proves the assertion.
\end{proof}

\begin{remark}\rm
In conditions of Corollary~\ref{C-new2}, if distributions ${\cal D}$ and ${\cal D}^\bot$ are totally umbilical, i.e.,
$\|H\|^2-\|h\|^2= \frac{d-1}{d}\,\|H\|^2$ and $\|H^\bot\|^2-\|h^\bot\|^2= \frac{d^\bot-1}{d^\bot}\,\|H^\bot\|^2$, then such a compact $(M,g)$ does not admit minimal isometric immersions ($\bar{H}=0$) in a Riemannian manifold $(\bar M,\bar g)$ with $\bar\delta^+_{{\rm m}}(d,d^\bot)<0$.
\end{remark}

\begin{corollary}\label{C-new3}
Let $f: (M,g; {\cal D}_1,{\cal D}_2,{\cal D}_3)\to(\bar M,\bar g)$ be an isometric immersion and $n_1+n_2+n_3=n$.
If $M$ is compact and all ${\cal D}_i$ are defined on an open set $M\setminus\Sigma$, $(k-1)(p-1)\ge 1$ and
$\|\,H_1+H_1^\bot+H_2+H_2^\bot+H_3+H_3^\bot\|\,\in L^p(M,g)$
then (for $Q_i$ given in Lemma~\ref{L-Qi})
\begin{eqnarray*}
 -\frac12\int_M (Q_1+Q_2+Q_3)\,d\vol_g
 \le \frac{1}{3}\int_M \|\,\bar{H}\,\|^2\,d\vol_g +\,\bar\delta^+_{{\rm m}}(n_1,n_2,n_3)\,{\rm Vol}(M,g).
\end{eqnarray*}
\end{corollary}

\begin{proof}
This follows from \eqref{eq-IF-3}, \eqref{E-ineq-k} and the Divergence-type theorem in Remark~\ref{R-3.5}.
\end{proof}

\begin{remark}\rm
In conditions of Corollary~\ref{C-new3}, if $h_i=h_i^\bot=0\ (i=1,2,3)$, then such a compact manifold $(M,g)$ does not admit minimal isometric immersions in a Riemannian manifold $(\bar M,\bar g)$ with $\bar\delta^+_{{\rm m}}(n_1,n_2,n_3)<0$.
\end{remark}

\begin{example}\rm
Let ${\cal D}_i\ (i=1,2,3)$ be 1-dimensional distributions orthogonal to
three pairwise orthogonal codimension-one foliations ${\mathcal F}_i$ on $(M^3,g)$, see Example~\ref{Ex-2.4}.
If $f: (M^3,g; {\cal D}_1,{\cal D}_2,{\cal D}_3)\to (\bar M, \bar g)$ is an isometric immersion, then $\tau\le\frac13\,\|\,\bar H\,\|^2 +\bar\delta^+_{\rm m}(1,1,1)$. Note that $\bar\delta^+_{\rm m}(1,1,1)(x)=\max\{\bar\tau(V): V\subset T_x\bar M,\ \dim V=3\}$.
Moreover, if foliations ${\mathcal F}_i\ (i=1,2,3)$ are minimal and not totally geodesic, then $(M^3,g)$ does not admit minimal isometric immersions in a Euclidean space.
\end{example}

Next, we consider the case when a distribution ${\cal D}$ is represented as the sum of two orthogonal distributions
of ranks $n_i>0$:
 ${\cal D}={\cal D}_1\oplus{\cal D}_2$,
thus, $n_1+n_2 = d$.

An isometric immersion $f: (M,g; {\cal D}_1,{\cal D}_2)\to(\bar M,\bar g)$
is called \textit{mixed totally geodesic} on ${\cal D}$ if
\[
 \bar{h}(X,Y)=0\quad{\rm for\ all}\ \  X\in{\cal D}_1,\ Y\in{\cal D}_2\,.
\]

\begin{theorem}\label{T-k}
Let $f: (M,g; {\cal D})\to(\bar M,\bar g; \bar{\cal D})$ be an adapted isometric immersion
and ${\cal D}={\cal D}_1\oplus{\cal D}_2$.
Then
\begin{equation}\label{E-ineq-k}
 {\rm S}_{\,\rm m}({\cal D}_1,{\cal D}_2) \le \frac14\,\|\,\bar H_{\cal D}\,\|^2
 +\bar\delta^+_{{\rm m},\bar{\cal D}}(n_1,n_2)\,.
\end{equation}
The equality in \eqref{E-ineq-k} holds at a point $x\in M$ if and only if $f$ is mixed totally geodesic on ${\cal D}_x$,
$\bar{H}_1(x)=\bar{H}_2(x)$ (the mean curvature vectors of ${\cal D}_i$)
and $\bar{\rm S}_{\,\rm m}({\cal D}_1(x),{\cal D}_2(x))=\bar\delta^+_{{\rm m},\bar{\cal D}}(n_1,n_2)(x)$.
\end{theorem}

\begin{proof} The proof of \eqref{E-ineq-k} is similar to the proof of Theorem~\ref{T-D}. We take $V_i={\cal D}_i(x)$.
The proof of the second assertion follows directly from the cases of equality, as in the proof of Theorem~\ref{T-D}.
\end{proof}

\begin{remark}\rm
Let $f: (M,g; {\cal D})\to(\bar M,\bar g; \bar{\cal D})$ be an adapted isometric immersion
and ${\cal D}={\cal D}_1\oplus{\cal D}_2$.
The~following counterpart of \eqref{E-ineq-k} is a special case of \cite[Eq.~(19)]{r1-2022}:
\begin{equation}\label{E-ineq-k3}
 {\rm S}_{\,\rm mix}({\cal D}_1,{\cal D}_2,{\cal D}^\bot) \le \frac13
 \,\|\,\bar H\,\|^2
 +\bar\delta^+_{\rm m}(n_1,n_2,d^\bot)\,.
\end{equation}
\end{remark}

\begin{corollary}[\rm for (i) see {\rm \cite[Corollary 6]{r1-2022}}]
Let $(M,g; {\cal D})$ be a sub-Riemannnian manifold with ${\cal D}={\cal D}_1\oplus{\cal D}_2$.

(i) if $\,{\rm S}_{\,\rm m}({\cal D}_1,{\cal D}_2)>0$, then $(M,g; {\cal D})$ does not admit
${\cal D}$-minimal isometric immersions in a Euclidean space.

(ii) if $\,{\rm S}_{\,\rm mix}({\cal D}_1,{\cal D}_2,{\cal D}^\bot)>0$, then $(M,g; {\cal D})$ does not admit
minimal isometric immersions in a Euclidean space.
\end{corollary}

\begin{proof}
This follows directly from \eqref{E-ineq-k} for (i) and from \eqref{E-ineq-k3} for (ii).
\end{proof}

\begin{corollary}
In conditions of Theorem~\ref{T-k},
let ${\cal D}_1$ be spanned by a unit vector field $N$. Then
\begin{equation}\label{E-ineq-k2}
 \Ric_{N,N} \le \frac{1}{4}\,\|\,\bar{H}_{\cal D}\,\|^2 + \bar r_{d-1\,|\,\bar{\cal D}}\,,
\end{equation}
where $d=\dim{\cal D}$
and $\bar r_{d-1\,|\,\bar{\cal D}}$ is the supremum of the $(d-1)$-th Ricci curvature of $(\bar M,\bar g)$ along $\bar{\cal D}$.
The equality in \eqref{E-ineq-k2} holds if and only if $f$ is mixed totally geodesic along ${\cal D}$,
$\bar{H}_1(x)=\bar{H}_2(x)$ and $\Ric_{N,N} =\bar r_{d-1}$ at each point~$x\in M$.
\end{corollary}

Applying \eqref{eq-IF-l2} to \eqref{E-ineq-k} on a compact manifold $M$, gives the following

\begin{corollary}
In conditions of Theorem~\ref{T-k}, let ${\cal D}_3={\cal D}^\bot$. If $M$ is compact
and all ${\cal D}_i$ are defined on an open set $M\setminus\Sigma$, $(k-1)(p-1)\ge 1$ and
$\|\,H_1+H_1^\bot+H_2+H_2^\bot-H_3-H_3^\bot\|\,\in L^p(M,g)$,
then (for $Q_i$ given in Lemma~\ref{L-Qi}),
\begin{eqnarray*}
 \int_M (Q_{3}-Q_{1}-Q_{2})
 \,d\vol_g
 \le \frac{1}{4}\int_M \|\,\bar{H}_{\cal D}\,\|^2\,d\vol_g +\,\bar\delta^+_{{\rm m},{\cal D}}(n_1,n_2)\,{\rm Vol}(M,g).
\end{eqnarray*}
\end{corollary}

\begin{proof}
This follows from \eqref{eq-IF-l2}, \eqref{E-ineq-k} and the Divergence-type theorem in Remark~\ref{R-3.5}.
\end{proof}

Finally, we apply $\delta_{\cal D}$-invariants \eqref{E-ineq1-D} to isometric immersions of sub-Riemannian manifolds.

\begin{theorem}
Let $f: (M,g; {\cal D})\to(\bar M,\bar g,\bar{\cal D})$ be an adapted isometric immersion.
Then for eny $k$-tuple $(n_1,\ldots,n_k)\in S(d)$ we get the inequality
\begin{eqnarray}\label{E-ineq-C}
 \delta_{\cal D}(n_1,\ldots,n_k)\le \frac{d+k-1-\sum_{\,i} n_i}{2(d+k-\sum_{\,i} n_i)}\,\|\,\bar{H}_{\cal D}\,\|^2
 +\frac12\,\big(d(d-1)- \sum\nolimits_{\,i} n_i(n_i-1)\big)\max \bar K_{\,|\,\bar{\cal D}}.
\end{eqnarray}
\end{theorem}

\begin{proof}
This is similar to the proof of \cite[Theorem~13.3]{chen1}.
\end{proof}

The case of equality in \eqref{E-ineq-C} is similar to \cite[Theorem~13.3: (a), (b)]{chen1}.
Extremal immersions in Euclidean space in terms of $\delta_{\cal D}$-invariants
are the sub-Riemannian analogue of Chen's ``ideal immersions".

\begin{corollary}
A sub-Riemannnian manifold $(M,g; {\cal D})$ with the condition $\delta_{\cal D}(n_1,\ldots,n_k)>0$
for some $(n_1,\ldots,n_k)\in S(d,k)$ does not admit ${\cal D}$-minimal
isometric immersions in a Euclidean~space.
\end{corollary}

\begin{proof}
This follows directly from \eqref{E-ineq-C}.
\end{proof}

\end{document}